\documentclass[12pt,leqno,oneside]{amsart}
\usepackage{amsmath,amssymb,amsthm}
\usepackage{tikz}
\usepackage[nohead,nofoot,margin=1.75in,centering]{geometry}

\newcommand{\Z}{{\mathbb Z}}

\newcommand{\B}{\mathfrak{B}}
\newcommand{\End}{\operatorname{End}}
\newcommand{\GL}{\operatorname{GL}}
\renewcommand{\O}{{\operatorname{O}}}
\newcommand{\Sp}{{\operatorname{Sp}}}
\newcommand{\Sym}{\mathfrak{S}}

\newcommand{\topp}{\tau}
\newcommand{\bott}{\beta}

\newcommand{\TO}{\dot{T}}

\newcommand{\Lamdot}{\Lambda_1}
\newcommand{\Lamddot}{\Lambda_2}

\newcommand{\ep}{\varepsilon}

\newtheorem{thm}{Theorem}
\newtheorem*{thm*}{Theorem}

\newtheorem*{lem*}{Lemma}
\newtheorem{prop}{Proposition}
\newtheorem*{prop*}{Proposition}

\newtheorem*{cor*}{Corollary}

\theoremstyle{definition}

\newtheorem*{defn*}{Definition}

\newtheorem*{example*}{Example}

\newtheorem*{examples*}{Examples}

\newtheorem*{rmk*}{Remark}

\numberwithin{equation}{section}
\allowdisplaybreaks

\begin{document}
\title{A characteristic-free decomposition of tensor space as a 
Brauer algebra module}
\author{S.R.~Doty} \address{Mathematics and Statistics, Loyola
  University Chicago, Chicago, Illinois 60626 U.S.A.}
\email{doty@math.luc.edu}

\begin{abstract}\noindent
We obtain a characteristic-free decomposition of tensor space,
regarded as a module for the Brauer centralizer algebra. 
\end{abstract}

\maketitle

\section*{Introduction}\noindent
The representation theory of a symmetric group $\Sym_r$ on $r$ letters
(see \cite{James, Fulton}) starts with the transitive permutation
modules $M^\lambda$ indexed by the partitions $\lambda$ of $r$.  For
any field $k$, tensor space $(k^n)^{\otimes r}$, regarded as a
$k\Sym_r$-module via the place permutation action, admits a direct sum
decomposition into a direct sum of the $M^\lambda$. (If $n<r$ then not
all of the $M^\lambda$ appear in the decomposition.) The purpose of
this paper is to give a similar characteristic-free decomposition of
tensor space $(k^n)^{\otimes r}$, regarded as a module for the Brauer
algebra. (Characteristic 2 is excluded from some results, in order to
avoid technicalities.)  The main results are summarized together in
Section \ref{sec:main} below, for the convenience of the reader.

A different characteristic-free decomposition of $(k^n)^{\otimes r}$
as a module for the Brauer algebra was previously obtained in
\cite{HP}, by working with the action defined in terms of the standard
bilinear form on $k^n$. By choosing a different bilinear form, we
obtain a more refined decomposition than that of \cite{HP}, in most
cases, which should give more information.

Our approach is motivated by Schur--Weyl duality (see \cite{Schur,
  Doty:SWD, BD, DDH, DH}), although its full generality is not used
here. All we need is the fact that the action of the Brauer algebra
commutes with that of a suitable classical group.

Our results provide new characteristic-free representations $N^\xi$ of
the Brauer algebra, indexed by partitions $\xi$, which may be regarded
as analogues of the classical transitive permutation modules for
symmetric groups. It is hoped that these representations may be of
some use in the study of the representation theory of Brauer algebras,
especially in the non-semisimple case, where little is known.

The paper is organized as follows. After summarizing the main results
in Section \ref{sec:main}, we recall the decomposition of tensor space
regarded as a module for the symmetric group in Section \ref{sec:Sym},
define the Brauer algebra and its action on tensors in Section
\ref{sec:Brauer}, and prove our results in Sections
\ref{sec:symmetric}, \ref{sec:symmetric-odd}, and \ref{sec:skew}.

\section{Main results}\label{sec:main}\noindent
Fix a field $k$ of characteristic different from $2$.  Tensor space
$(k^n)^{\otimes r}$ is regarded as a module for the Brauer algebra
$\B_r(n)$ via an action (see Section \ref{sec:Brauer}) defined by the
nondegenerate symmetric bilinear form $(\ ,\ )$ on $k^n$ such that
$(e_i,e_{j'})=\delta_{ij}$, where $e_1, \dots, e_n$ is the standard
basis of $k^n$ and $j' = n+1-j$. This choice of bilinear form is
important for our results.

We will need the set $\Lambda(n,r)$ of $n$-part compositions of $r$,
defined by
\[
\Lambda(n,r) = \{(\lambda_1, \dots, \lambda_n) \in \Z^n \colon 0 \le
\lambda_i\ (\forall i),\ \lambda_1 + \cdots + \lambda_n = r\}. 
\]
For a given positive integer $l$, let sets $\Lambda_1(l,r)$ and
$\Lambda_2(l,r)$ be defined as follows:
\begin{align*}
\Lambda_1(l,r) &= \{(\xi_1, \dots, \xi_l) \in \Z^l \colon 
  |\xi_1| + \cdots + |\xi_l| = r-s,\ 0\le s \le r \}\\
\Lambda_2(l,r) &= \{(\xi_1, \dots, \xi_l) \in \Z^l \colon
       |\xi_1| + \cdots + |\xi_l| = r-2s,\ 0\le 2s \le r\}.
\end{align*}
If $n=2l+1$, we have a surjective map $\pi \colon \Lambda(n,r) \to
\Lambda_1(l,r)$ given by the rule
\[
   \pi(\lambda_1, \dots, \lambda_n) = (\lambda_1 - \lambda_{1'}, \dots,
   \lambda_l - \lambda_{l'}).
\]
If $n=2l$, the same rule defines a surjective map $\pi \colon
\Lambda(n,r) \to \Lambda_2(l,r)$. In either case, the fibers of $\pi$
determine the desired decomposition of tensor space. See Sections
\ref{sec:symmetric}, \ref{sec:symmetric-odd} for combinatorial
descriptions of the fibers, depending on the parity of $n$.

The well known characteristic-free decomposition of $(k^n)^{\otimes
  r}$ as a $k\Sym_r$-module, where $\Sym_r$ is the symmetric group on
$r$ letters, is given by
\[
  (k^n)^{\otimes r} = \textstyle\bigoplus_{\lambda \in \Lambda(n,r)} M^\lambda ,
\]
where $M^\lambda$ is a transitive permutation module for $k\Sym_r$,
realized as the $k$-span of all simple tensors $e_{i_1} \otimes \cdots
\otimes e_{i_r}$ of weight $\lambda$, with $\Sym_r$ acting by place
permutation on the simple tensors.  Given $\xi \in \Lambda_j(n,r)$ for
$j=1,2$ we define $N^\xi := \bigoplus_{\lambda \in \pi^{-1}(\xi)}
M^\lambda$. Then we prove the following:

\begin{thm}\label{thm:1}
  A characteristic-free decomposition (for characteristic $k \ne 2$)
  of $(k^n)^{\otimes r}$ as a $\B_r(n)$-module is given
  by \[(k^n)^{\otimes r} = \textstyle\bigoplus_{\xi \in
    \Lambda_j(l,r)} N^\xi \qquad\quad (j=1,2),\] where $j=1$ if
  $n=2l+1$ and $j=2$ if $n=2l$.
\end{thm}

This is nothing but a weight space decomposition of $(k^n)^{\otimes
  r}$, regarded as a module for the diagonal torus in the orthogonal
group $\O_n(k)$, the group of matrices preserving the bilinear form
$(\ ,\ )$ on $k^n$. The proof, which is given in Sections
\ref{sec:symmetric} and \ref{sec:symmetric-odd}, is almost trivial:
the main idea is just the well known fact that the actions of
$\O_n(k)$ and $\B_r(n)$ on $k^n$ commute.

The bilinear form is chosen so that the diagonal tori in $\GL_n(k)$
and $\O_n(k)$ are compatible upon restriction from $\GL_n(k)$ to
$\O_n(k)$, which is what makes everything work. (It is well known that
in Lie theory, our choice of defining form for the orthogonal group,
or one very similar to it, is more natural than the standard defining
form.)

\medskip

In case $n=2l$, we obtain another characteristic-free decomposition of
$(k^n)^{\otimes r}$, with no restriction on the characteristic of $k$,
by replacing the role of the orthogonal group in the above by the
symplectic group $\Sp_{n}(k)$, defined as the set of matrices
preserving the skew-symmetric bilinear form $(\ , \ )$ on $k^n$ given
by $(e_i, e_{j'}) = \varepsilon_i \delta_{i,j}$, where $\varepsilon_j
= 1$ if $j<j'$ and $-1$ otherwise. There is an action of $\B_r(-n)$ on
tensor space $(k^n)^{\otimes r}$, defined in terms of the bilinear
form.

\begin{thm}\label{thm:2}
  If $n=2l$, a characteristic-free decomposition of $(k^n)^{\otimes
    r}$ as a $\B_r(-n)$-module is given by \[(k^n)^{\otimes r} =
  \textstyle\bigoplus_{\xi \in \Lambda_2(l,r)} N^\xi.\] 
\end{thm}

Again, this is just a weight space decomposition for the torus of
diagonal matrices in $\Sp_n(k)$.  The proof in this case is quite
similar to the even orthogonal case, and is sketched in Section
\ref{sec:skew}.

\medskip

These results provide a new family $\{ N^\xi \}$ of characteristic
free representations of the Brauer algebra, indexed by $\xi \in
\Lambda_1(l,r)$ or $\Lambda_2(l,r)$. Actually, we show that the
hyperoctahedral group $(\Z/2\Z)^l \rtimes \Sym_l$ acts naturally on
either of the sets $\Lambda_1(l,r)$ or $\Lambda_2(l,r)$ through signed
permutations of the entries of a weight, and modules $N^\xi$ indexed
by weights in the same orbit are all isomorphic, so it suffices to
restrict one's attention to the modules $N^\xi$ indexed by the
dominant weights $\xi$, which are partitions. So, up to isomorphism,
the Brauer algebra direct summands of tensor space are indexed by the
subsets $\Lambda^+_1(l,r)$ or $\Lambda^+_2(l,r)$ of partitions in
$\Lambda_1(l,r)$ or $\Lambda_2(l,r)$, respectively.

The modules $N^\xi$ are defined by gluing various permutation modules
$M^\lambda$ together. The analysis in Sections \ref{sec:symmetric} and
\ref{sec:skew} reveal that when $n =2l$ and $\xi = (\xi_1, \dots,
\xi_l)$ is a partition of $r-2s$ into not more than $l$ parts, for $0
\le 2s \le r$, the various $\lambda$ in the fiber $\pi^{-1}(\xi)$ are
precisely the weights of the form
\[
   (\xi+\nu)\parallel \nu^* \qquad\text{for } \nu \in
\Lambda(l,s)
\]
where $\nu^* = (\nu_s, \dots, \nu_1)$ is the \emph{reverse}
of $\nu = (\nu_1, \dots, \nu_s)$ and where $\parallel$ denotes
concatenation of finite sequences: 
\[
  (a_1, \dots, a_i) \parallel (b_1, \dots, b_j) := (a_1, \dots, a_i,
b_1, \dots b_j).
\]
Hence, in this case $N^\xi$ is the direct sum of $|\Lambda(l,s)|$
permutation modules. In particular, if $s=0$ there is just one
permutation module in $N^\xi$.

In case $n = 2l+1$ the analysis in Section \ref{sec:symmetric-odd}
reveals that if $\xi = (\xi_1, \dots, \xi_l)$ is a given partition of
$r-s$ into not more than $l$ parts, where $0 \le s \le r$, the various
$\lambda$ in the fiber $\pi^{-1}(\xi)$ are precisely the weights of
the form
\[
   (\xi+\nu) \parallel (s-2t)\parallel \nu^* \qquad
\text{for }\nu \in \Lambda(l,t)
\]
as $t$ varies over all possibilities in the range $0 \le 2t \le s$.
Hence, in this case $N^\xi$ is the direct sum of $\sum_{0 \le 2t \le
  s} |\Lambda(l,t)|$ permutation modules.

\section{Symmetric group decomposition of $(k^n)^{\otimes r}$}
\label{sec:Sym}\noindent
Let $k$ be an arbitrary field.  Consider an $n$-dimensional vector
space $k^n$ and its associated group $\GL_n(k)$ of linear
automorphisms. The group acts naturally on the space, and thus also
acts naturally on the $r$-fold tensor product $(k^n)^{\otimes r}$, via
the `diagonal' action:
\begin{equation}\label{eq:GL-action}
  g \cdot (v_1 \otimes \cdots \otimes v_r) = (g \cdot v_1) \otimes
  \cdots \otimes (g \cdot v_r).
\end{equation}
The symmetric group $\Sym_r$ also acts on the right on
$(k^n)^{\otimes r}$, via the so-called `place permutation' action,
which satisfies
\begin{equation}\label{eq:Sym-action}
 (v_1 \otimes \cdots \otimes v_r) \cdot \pi = v_{(1)\pi^{-1}} \otimes
  \cdots \otimes v_{(r)\pi^{-1}}.
\end{equation}
Notice that we adopt the convention that elements of $\Sym_r$ act on
the right of their arguments. Now it is clear from the definitions
that the actions of these two groups commute:
\[
  g \cdot \big( (v_1 \otimes \cdots \otimes v_r) \cdot \pi \big) =
  \big( g \cdot (v_1 \otimes \cdots \otimes v_r) \big) \cdot \pi,
\]
for all $g \in \GL_n(k)$, $\pi \in \Sym_r$.

In order to simplify the notation, we put $V:= k^n$.  We more or less
follow Section 3 of \cite{Green}. The group $\GL_n(k)$ contains an
abelian subgroup $T$ consisting of the diagonal matrices in
$\GL_n(k)$, and this subgroup (being abelian) must act semisimply on
$V^{\otimes r} = (k^n)^{\otimes r}$. This leads in the usual way to a
`weight space' decomposition
\begin{equation}\label{eq:wtspace}
  V^{\otimes r} = \textstyle \bigoplus_{\lambda\in X(T)} V^{\otimes
    r}_\lambda,
\end{equation}
where $\lambda$ varies over the group $X(T)$ of characters $\lambda
\colon T \to k^\times$, and where the weight space $V^{\otimes
  r}_\lambda$ is the linear span of the tensors $v=v_1 \otimes \cdots
\otimes v_r$ such that $t \cdot v = \lambda(t) v$, for all $t \in T$.

Clearly $T$ is isomorphic to the direct product $(k^\times)^n$ of $n$
copies of the multiplicative group $k^\times$ of the field.  Let
$\ep_i \in X(T)$ be evaluation at the $i$th diagonal entry of an
element of $T$. Regarding the abelian group $X(T)$ as an additive
group as usual, observe that $\ep_1, \dots, \ep_n$ is a basis for
$X(T)$, and thus the map $\Z^n \to X(T)$ given by $(\lambda_1, \dots,
\lambda_n) \mapsto \sum_i \lambda_i \ep_i$ is an isomorphism. So we
identify $X(T)$ with $\Z^n$ by means of this isomorphism.

The direct sum in \eqref{eq:wtspace} is formally taken over $X(T)$;
however, many of the summands are actually zero.  It is easy to check
that the weight space decomposition of $V$, regarded as a $T$-module,
is given by $V = V_{\ep_1} \oplus \cdots \oplus V_{\ep_n}$. It follows
immediately that the set of weights of $V^{\otimes r}$ is the set
\[
  \Lambda(n,r) = \{ (\lambda_1, \dots \lambda_n) \in \Z^n \colon
  \lambda_i \ge 0, \lambda_1 + \cdots + \lambda_n = r \}
\]
of $n$-part compositions of $r$, under the isomorphism of $X(T)$ with
$\Z^n$. Thus we may write \eqref{eq:wtspace} in the better form 
\begin{equation} \label{eq:wtspaceLambda}
  V^{\otimes r} = \textstyle \bigoplus_{\lambda\in \Lambda(n,r)} M^\lambda 
\end{equation}
where we have, partly to simplify notation but also to serve
tradition, put $M^\lambda := V^{\otimes r}_\lambda$. Since the actions
of $\Sym_r$ and $\GL_n(k)$ commute, each $M^\lambda$ is a
$k\Sym_r$-module, so \eqref{eq:wtspaceLambda} gives a decomposition of
$V^{\otimes r}$ as $k\Sym_r$-modules.

Let us describe the vector space $M^\lambda$ in greater detail.  Let
$e_1, \dots, e_n$ be the standard basis of $V = k^n$. Then
$M^\lambda$, for any $\lambda \in \Lambda(n,r)$, has a basis given by
the set of simple tensors $e_{i_1} \otimes \cdots \otimes e_{i_r}$
such that in the multi-index $(i_1, \dots, i_r)$ there are exactly
$\lambda_1$ occurrences of $1$, $\lambda_2$ occurrences of $2$, and so
forth. Evidently the action of the symmetric group $\Sym_r$ permutes
such simple tensors transitively, so $M^\lambda$ is in fact a
transitive permutation module. The representation theory of $\Sym_r$
over $k$ starts with these permutation modules (see e.g.\ \cite{James,
  Fulton}) usually defined rather differently. At this point we could
introduce row standard tableaux of shape $\lambda$ (or, equivalently,
the ``tabloids'' of \cite{James}) to label our basis elements of
$M^\lambda$, but we shall have no need of such combinatorial gadgets.

The symmetric group $\Sym_n$ can be identified with the Weyl group $W$
of $\GL_n(k)$. (Recall that the theory of BN-pairs (due to J.~Tits)
can be used to define $W$ in any $\GL_n(k)$ by a uniform method,
including the case when $k$ is finite.)  The group $W$ may be
identified with the subgroup of permutation matrices of $\GL_n(k)$, so
it acts naturally (on the left) on $V^{\otimes r}$ by restriction of
the action of $\GL_n(k)$. Moreover, $W=\Sym_n$ acts on the set $\Z^n$
by
\begin{equation}
  w^{-1}(\lambda_1, \dots, \lambda_n) =
  (\lambda_{w(1)}, \dots, \lambda_{w(n)}). 
\end{equation}
This action stabilizes the set $\Lambda(n,r)$, so we have also an
action of $W$ on $\Lambda(n,r)$. Each $W$-orbit of $\Lambda(n,r)$
contains exactly one \emph{dominant} weight: a weight $\lambda =
(\lambda_1, \dots, \lambda_n)$ such that $\lambda_1 \ge \lambda_2 \ge
\cdots \ge \lambda_n$. Denote the set of dominant weights in
$\Lambda(n,r)$ by $\Lambda^+(n,r)$. This set may be identified with
the set of partitions of $r$ into not more than $n$ parts. The
following is immediate from Proposition (3.3a) of \cite{Green}.

\begin{prop}\label{prop:1}
  For any $w \in W$, the right $k\Sym_r$-modules $M^\lambda$ and
  $M^{w(\lambda)}$ are isomorphic. 
\end{prop}

\begin{proof}
The isomorphism is given on basis elements by mapping a simple tensor
$e_{i_1} \otimes \cdots \otimes e_{i_r}$ of weight $\lambda$ to the
simple tensor $e_{w(i_1)} \otimes \cdots \otimes e_{w(i_r)}$ of weight
$w(\lambda)$.
\end{proof}

Thus, when considering the $M^\lambda$, we may as well confine our
attention to the ones labeled by dominant weights (i.e.\ partitions)
$\lambda \in \Lambda^+(n,r)$.

\section{The Brauer algebra}\label{sec:Brauer}\noindent
A Brauer $r$-diagram (introduced in \cite{Brauer}) is an undirected
graph with $2r$ vertices and $r$ edges, such that each vertex is the
endpoint of precisely one edge. By convention, such a graph is usually
drawn in a rectangle with $r$ vertices each equally spaced along the
top and bottom edges of the rectangle. For example, the picture below
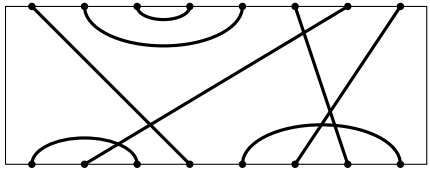
\[ 
\begin{tikzpicture}[scale=0.7]
  \draw (0,0) rectangle (8,3);
  \foreach \x in {0.5,1.5,...,7.5}
    {\fill (\x,3) circle (2pt);
     \fill (\x,0) circle (2pt);}
  \begin{scope}[very thick]
    \draw (0.5,3) -- (3.5,0);
    \draw (2.5,3) arc (180:360:0.5 and 0.25);
    \draw (1.5,3) arc (180:360:1.5 and 0.75);
    \draw (2.5,0) arc (0:180:1 and 0.5);
    \draw (6.5,3) -- (1.5,0);
    \draw (7.5,0) arc (0:180:1.5 and 0.75);
    \draw (5.5,3) -- (6.5,0);
    \draw (7.5,3) -- (5.5,0);
  \end{scope}
\end{tikzpicture}
\]
depicts a Brauer $8$-diagram.  Let $k$ be an arbitrary field. Let
$\B_r(\pm n)$ be the vector space over $k$ with basis the
$r$-diagrams, where we assume that $n$ is even in the negative
case. Brauer defined a natural multiplication of $r$-diagrams such
that $\B_r(\pm n)$ becomes an associative algebra.  In order to
describe the multiplication rule, it is convenient to introduce the
notations $\topp(d)$ and $\bott(d)$ for the sets of vertices along the
top and bottom edges of a diagram $d$.  Then the multiplication rule
works as follows. Given $r$-diagrams $d_1$ and $d_2$, place $d_1$
above $d_2$ and identify the vertices in $\bott(d_1)$ in order with
those in $\topp(d_2)$. The resulting graph consists of $r$ paths whose
endpoints are in $\topp(d_1) \cup \bott(d_2)$, along with a certain
number, say $s$, of cycles which involve only vertices in the middle
row. Let $d$ be the $r$-diagram whose edges are obtained from the
paths in this graph. Then the product of $d_1$ and $d_2$ in $\B_r(\pm
n)$ is given by $d_1 d_2 = (\pm n \cdot 1_k)^s d$.

Now we describe a right action of $\B_r(\pm n)$ on $V^{\otimes r}$,
which depends on the defining bilinear form $(\ ,\ )$. This is the
symmetric form defined in Section \ref{sec:main} in the positive case
and the skew-symmetric form defined in Section \ref{sec:main} in the
negative case.  We always assume characteristic $k \ne 2$ in the
symmetric case. We let $e^*_1, \dots, e^*_n$ be the basis dual to the
standard basis $e_1, \dots, e_n$ of $k^n$ with respect to the bilinear
form, in either case, so that $(e_i, e^*_j) = \delta_{ij}$. Given any
$r$-diagram $d$, let $(d)\varphi$ be the matrix whose $(\underline{i},
\underline{j})$-entry, for $\underline{i} = (i_1, \dots, i_r)$,
$\underline{j}=(j_1,\dots, j_r)$ is determined by the following
procedure:
\begin{enumerate}
\item Label the vertices along the top edge of $d$ from left to right
  with $e_{i_1}, \dots, e_{i_r}$ and label the vertices along the
  bottom edge from left to right with $e^*_{j_1}, \dots, e^*_{j_r}$. 

\item The $(\underline{i}, \underline{j})$-entry of $(d)\varphi$ is
  the product of the values $(u,v)$ over the edges $\ep$ of $d$, where
  for each edge, $u$ and $v$ are the labels on its vertices, ordered
  so that a vertex in $\topp(d)$ precedes one in $\bott(d)$, and from
  left to right within $\topp(d)$ and $\bott(d)$.
\end{enumerate}
This determines the desired action: let $d$ act on $V^{\otimes r}$ as
the linear endomorphism determined by the matrix $(d)\varphi$. Then
$\varphi$ extends linearly to a representation $\varphi\colon
\B_r(\pm n)^{opp} \to \End_k(V^{\otimes r})$.

It will be useful to have a better understanding of the action of
$\B_r(\pm n)$. For this, observe that the $r$-diagrams in which every
edge connects a vertex in the top row to a vertex in the bottom row
correspond to permutations in $\Sym_r$, and their action on
$V^{\otimes r}$ is the same as that defined by \eqref{eq:Sym-action}.
Let us agree to call such diagrams \emph{permutation} diagrams.  Since
we write maps in $\Sym_r$ on the right of their arguments, the
multiplication of diagrams defined above corresponds to composition of
permutations, when restricted to such diagrams. Thus we have a
subalgebra of $\B_r(\pm n)$, namely the subalgebra spanned by the
permutation diagrams, isomorphic to $k\Sym_r$, and this subalgebra
acts on $V^{\otimes r}$ via the usual place-permutation action,
independently of the choice of defining bilinear form $(\ , \ )$.

Now let $c_0$ be the unique $r$-diagram in which the first two
vertices in $\topp(c_0)$ are joined by an edge, and similarly for the
first two vertices in $\bott(c_0)$, with the $j$th vertex in
$\topp(c_0)$ joined to the $j$th vertex in $\bott(c_0)$ for $j=3,
\dots, r$. For instance, in case $r = 8$ the diagram $c_0$
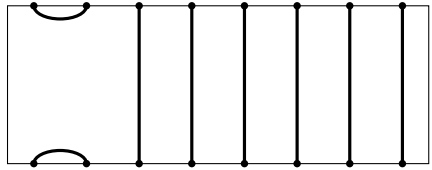
\[ 
\begin{tikzpicture}[scale=0.7]
  \draw (0,0) rectangle (8,3);
  \foreach \x in {0.5,1.5,...,7.5}
    {\fill (\x,3) circle (2pt);
     \fill (\x,0) circle (2pt);}
  \begin{scope}[very thick]
    \draw (0.5,3) arc (180:360:0.5 and 0.25);
    \draw (1.5,0) arc (0:180:0.5 and 0.25);
    \draw (2.5,3) -- (2.5,0);
    \draw (3.5,3) -- (3.5,0);
    \draw (4.5,3) -- (4.5,0);
    \draw (5.5,3) -- (5.5,0);
    \draw (6.5,3) -- (6.5,0);
    \draw (7.5,3) -- (7.5,0);
  \end{scope}
\end{tikzpicture}
\]
is the diagram pictured above.  It is well known (see \cite{Brown})
that $\B_r(\pm n)$ is generated by the permutation diagrams together
with the diagram $c_0$.  This may be argued as follows. Call an edge
in a diagram $d$ \emph{horizontal} if its endpoints both lie in
$\topp(d)$, or both lie in $\bott(d)$. The number of horizontal edges
in the top edge of the enclosing rectangle must equal the number in
the bottom edge.  Put $B_j$ equal to the span of the diagrams with
exactly $2j$ horizontal edges. Then, as a vector space, $\B_r(\pm n) =
B_0 \oplus B_1 \oplus \cdots \oplus B_m$, where $m = \lfloor r/2
\rfloor$, the integer part of $r/2$. Clearly $B_0 = k\Sym_r$. By
acting on $c_0$ on the left or right by permutations, one can generate
$B_1$. Then by picking diagrams in $B_1$ appropriately, one may obtain
a diagram in $B_2$, and thus obtain all diagrams in $B_2$ by again
acting by permutations on the left and right. Continuing in this way,
one eventually generates all diagrams in the algebra.

Thus, in order to unambiguously specify the action of the full algebra
$\B_r(\pm n)$ on $V^{\otimes r}$, we only need to see how the diagram
$c_0$ acts.  This depends on the choice of the defining bilinear form
$(\ , \ )$, and by direct computation we see that in the symmetric
case $c_0$ acts by the rule
\begin{equation}
  \textstyle (e_{i_1} \otimes \cdots \otimes e_{i_r}) \cdot c_0 =
  \delta_{i_1,i'_2} \sum_{j=1}^n e_j \otimes e_{j'} \otimes e_{i_3}
  \otimes \cdots \otimes e_{i_r}.
\end{equation}
In the skew-symmetric case $c_0$ acts by the rule
\begin{equation}
  \textstyle (e_{i_1} \otimes \cdots \otimes e_{i_r}) \cdot c_0 =
  \delta_{i_1,i'_2} \sum_{j=1}^n \ep_j\, e_j \otimes e_{j'} \otimes
  e_{i_3} \otimes \cdots \otimes e_{i_r}.
\end{equation}
These actions are closely related to Weyl's contraction operators in
\cite{Weyl}.

\section{The $\B_r(n)$ decomposition of $(k^n)^{\otimes r}$
in the\\symmetric case, where $n=2l$}\label{sec:symmetric}\noindent 
From now on, until further notice, we assume that the field $k$ has
characteristic not $2$. This avoids technicalities pertaining to the
definition of orthogonal groups over fields of characteristic $2$.  We
define $\O_n(k)$ to be the group of isometries of $V$ with respect to
the symmetric form $(\ ,\ )$ given in Section \ref{sec:main}. Then the
action of $\B_r(n)$ on tensor space $V^{\otimes r} = (k^n)^{\otimes
  r}$, defined in the preceding section, commutes with the natural
action of $\O_n(k)$ (given by restricting the action of $\GL_n(k)$).

Let $\TO$ be the abelian subgroup of $\O_n(k)$ consisting of the
diagonal matrices in $\O_n(k)$. Thus, a diagonal matrix
$\text{diag}(t_1, \dots, t_n) \in \GL_n(k)$ belongs to $\TO$ if and
only if 
\begin{equation}\label{eq:TO}
t_i t_{i'} = 1 \quad\text{ for all $i = 1, \dots, n$.} 
\end{equation}
It will be useful to separate the consideration of the cases where $n$
is even and odd, so we assume that $n=2l$ for the remainder of this
section, and consider the odd case in the next section.

The description of $\TO$ in \eqref{eq:TO} shows in this case that
$\TO$ is isomorphic to the direct product $(k^\times)^l$ of $l = n/2$
copies of the multiplicative group $k^\times$ of the field $k$.  So
the character group $X(\TO)$ is isomorphic to $\Z^l$, so we identify
$X(\TO)$ with $\Z^l$.

There is a group homomorphism $\pi \colon X(T) \to X(\TO)$ given by
restriction: $\pi(\lambda) = \lambda_{\mid \TO}$ for $\lambda \in
X(T)$. Since $\TO \subset T$, given a character $\xi \in
X(\TO)$, one can extend it to a character $\lambda \in X(T)$ such
that $\lambda_{\mid \TO} = \xi$. It follows that the map $\pi$
is surjective.

In terms of the identifications $X(T) = \Z^n$ and $X(\TO) = \Z^l$,
the map $\pi$ is given by the rule
\[
  (\lambda_1, \dots, \lambda_n) \mapsto (\lambda_1 -
   \lambda_{1'}, \dots, \lambda_{l}-\lambda_{l'}). 
\]
We next consider how to characterize the image $\Lamddot(l,r)$ of the
set $\Lambda(n,r)$ under the map $\pi$.

\begin{prop}\label{prop:2}
  When $n = 2l$, the image $\Lamddot(l,r)$ of the set $\Lambda(n,r)$
  under $\pi$ is the set of all $\xi = (\xi_1, \dots, \xi_l) \in \Z^l$
  such that $|\xi_1| + \cdots + |\xi_l|=r-2s$, where $0 \le 2s \le r$.
\end{prop}

\begin{proof}
If $\xi = \pi(\lambda)$ for $\lambda \in \Lambda(n,r)$ then
$|\xi_1|+\cdots +|\xi_l|$ satisfies the condition
\[
  |\xi_1|+\cdots +|\xi_l| = \epsilon_1(\lambda_1-\lambda_{1'}) +
  \cdots + \epsilon_l(\lambda_l-\lambda_{l'})
\]
where for each $i=1, \dots, l$ the sign $\epsilon_i$ is defined to be
$1$ if $\lambda_i \ge \lambda_{i'}$ and $-1$ otherwise. This is just a
signed sum of the parts of $\lambda$, so is congruent modulo 2 to the
sum of the parts of $\lambda$. Thus $|\xi_1|+\cdots +|\xi_l| = r-2s$
for some $s\in \Z$. Clearly $0 \le 2s \le r$. This proves the
necessity of the condition for membership in $\Lamddot(l,r)$.

It remains to prove the sufficiency of the condition. Given $\xi \in
\Z^l$ satisfying the condition $|\xi_1| + \cdots + |\xi_l|=r-2s$,
where $0 \le 2s \le r$, we define a corresponding $\mu \in
\Lambda(n,r-2s)$ as follows: put $\mu_i = \xi_i$ if $\xi_i > 0$, put
$\mu_{i'} = -\xi_i$ if $\xi_i < 0$, and put all the other entries of
$\mu = (\mu_1, \dots, \mu_n)$ to zero. Now pick $\nu \in \Lambda(l,s)$
arbitrarily. Then let $\lambda$ be obtained from $\mu$ and $\nu$ by
adding the parts of $\nu$ in order to $(\mu_1, \dots, \mu_l)$ and by
adding the parts of $\nu$ in reverse order to $(\mu_{l+1}, \dots,
\mu_{2l})$, so that
\[
 \lambda = (\mu_1+\nu_1, \dots, \mu_l+\nu_l, \mu_{l+1}+\nu_l,
 \dots, \mu_{2l}+\nu_1).
\]
Then it easily checked that $\pi(\lambda) = \xi$. 
\end{proof}

For each $\xi \in \Lamddot(l,r)$, the proof of the preceding
proposition reveals an algorithm for writing down the members of the
fiber $\pi^{-1}(\xi)$, and in particular shows that the cardinality of
the fiber is $|\Lambda(l,s)|$, where $s$ is as above.  By grouping
terms in the direct sum decomposition \eqref{eq:wtspaceLambda}
according to the fibers we obtain
\begin{equation}\label{eq:wtspace2ell} \textstyle
   (k^n)^{\otimes r} = V^{\otimes r} = \bigoplus_{\xi \in
    \Lamddot(l,r)} \big(\bigoplus_{\lambda \in \pi^{-1}(\xi)}
  M^\lambda \big) = \bigoplus_{\xi \in \Lamddot(l,r)} N^{\xi}
\end{equation}
where we define $N^{\xi}$ for any $\xi \in \Lamddot(l,r)$ by
$N^{\xi}:= \bigoplus_{\lambda \in \pi^{-1}(\xi)} M^{\lambda}$.

The $N^{\xi}$ are just the weight spaces under the action of the
abelian group $\TO$, so \eqref{eq:wtspace2ell} gives the weight space
decomposition of tensor space as a $\TO$-module.

Since the actions of $\O_n(k)$ and $\B_r(n)$ commute, it is clear
that each weight space $N^{\xi}$ for $\xi\in \Lamddot(l,r)$ is
a right $\B_r(n)$-module. Hence \eqref{eq:wtspace2ell} is a
decomposition of tensor space $(k^n)^{\otimes r}$ as a
$\B_r(n)$-module, and we have achieved our goal in the case $n =
2l$.

It remains to notice some isomorphisms existing among the
$\B_r(n)$-modules $N^{\xi}$.

As we already pointed out, the Weyl group $W$ associated to $\GL_n(k)$
acts on $\{ M^\lambda \colon \lambda \in \Lambda(n,r) \}$, and the
orbits are isomorphism classes. It can be expected that the Weyl
group associated to $\O_n(k)$ similarly acts on $\{ N^{\xi} \colon
\xi \in \Lamddot(l,r)) \}$, and again the orbits will be
isomorphism classes. 

The Weyl group $\dot{W}$ of $\O_n(k)$ is isomorphic to the semidirect
product $\{\pm 1\}^l \rtimes \Sym_l$, the group of signed
permutations on $l$ letters. We can realize $\dot{W}$ as a subgroup
of $\O_n(k)$, simply by taking the intersection of $W$ (the Weyl group
of $\GL_n(k)$, realized as the $n \times n$ permutation matrices) with
$\O_n(k)$.  A given $w \in W$ lies within this intersection if and only
if the condition $(e_{w^{-1}(i)} , e_{w^{-1}(j)}) = (e_i, e_j)$ holds
for all $i,j$. Thus, $\dot{W}$ is the set of $w \in W$ such that
\begin{equation}\label{eq:W-criterion}
  \delta_{w^{-1}(i), w^{-1}(j)'} = \delta_{i,j'} \qquad \text{for all }
  i,j = 1, \dots, n.
\end{equation}
It is easy to check by direct calculation that for any given $\sigma \in
\Sym_l$, if we define a corresponding $w_\sigma\in W$ such that
\[
  w_\sigma(i) = 
  \begin{cases}
  \sigma(i) & \text{if } 1 \le i \le l\\
  \sigma(i') & \text{if } l+1 \le i \le 2l  
  \end{cases}
\]
then $\sigma$ satisfies the condition
\eqref{eq:W-criterion}. Furthermore, the transposition $\tau_i$ that
interchanges $i$ with $i'$ also satisfies \eqref{eq:W-criterion}, and
thus $\dot{W}$ may be identified with the subgroup of $W$ generated by
the $w_\sigma$ ($\sigma \in \Sym_l$) and the $\tau_i$ ($i = 1,
\dots, l$).

This subgroup acts on $\Lambda(n,r)$ by restriction of the action of
$W$. This induces a corresponding action of $\dot{W}$ on the set
$\Lamddot(l,r)$, such that $\dot{w}(\xi) = \pi( w(\lambda) )$
if $w \in W$ corresponds to $\dot{w} \in \dot{W}$ and $\xi =
\pi(\lambda)$.  Since $\tau_i$ sends $\xi = (\xi_1, \dots,
\xi_l)$ to $(\xi_1, \dots, \xi_{i-1}, -\xi_i,
\xi_{i+1}, \dots, \xi_l)$, and $w_\sigma$ sends $\xi$
to $\sigma(\xi) = (\xi_{\sigma^{-1}(1)}, \dots,
\xi_{\sigma^{-1}(l)})$, it follows that $\dot{W}$ acts on the
set $\Lamddot(l,r)$ by signed permutations.

Thus, a fundamental domain for this action is the set
$\Lamddot^+(l,r)$ consisting of all $\xi \in \Lamddot(l,r)$
such that $\xi_1 \ge \xi_2 \ge \cdots \ge \xi_l \ge 0$. We call
elements of this set \emph{dominant} orthogonal weights. So, in
other words, each orbit of $\Lamddot(l,r)$ contains a unique
dominant orthogonal weight. Notice that a dominant orthogonal weight
is the same as a partition of not more than $l$ parts.

\begin{prop}\label{prop:3}
  For any $\dot{w} \in \dot{W}$, $\xi \in \Lamddot(l,r)$, the
  right $\B_r(n)$-modules $N^{\xi}$ and $N^{\dot{w}(\xi)}$ are
  isomorphic.
\end{prop}

\begin{proof}
This is similar to the proof of Proposition \ref{prop:1}.  The
isomorphism is given on basis elements by mapping a simple tensor
$e_{i_1} \otimes \cdots \otimes e_{i_r}$ of weight $\lambda \in
\pi^{-1}(\xi)$ to the simple tensor $e_{w(i_1)} \otimes \cdots
\otimes e_{w(i_r)}$ of weight $w(\lambda)$, where $w \in W$
corresponds to $\dot{w}$. Since $\pi(w(\lambda)) =
\dot{w}(\pi(\lambda))$ and the above holds for every $\lambda \in
\pi^{-1}(\xi)$, the result follows.
\end{proof}

Hence, when studying properties of the modules $N^{\xi}$, we may
as well confine our attention to the ones indexed by dominant
orthogonal weights; i.e., partitions. In the decomposition
\eqref{eq:wtspace2ell} each summand is isomorphic to some
$N^{\xi}$ for some $\xi$ such that $\xi$ is a partition of
$r - 2s$ into not more than $l$ parts, for some non-negative integer
$s \le r/2$. It is easy to see that all such possibilities actually
occur as direct summands in \eqref{eq:wtspace2ell}.

\section{The $\B_r(n)$ decomposition of $(k^n)^{\otimes r}$
in the\\symmetric case, where $n=2l+1$}\label{sec:symmetric-odd}\noindent 
Now we consider the case where $n = 2l+1$, still with the symmetric
bilinear form.  In this case, we have $(l+1)' = l+1$.  Thus, if a
diagonal matrix $\text{diag}(t_1, \dots, t_n) \in \GL_n(k)$ belongs to
$\TO$ then we necessarily have $t^2_{l+1} = 1$, and $t_{i'} =
t_i^{-1}$ for all $i \ne l+1$. Hence, the description of $\TO$ in
\eqref{eq:TO} shows in this case that $\TO$ is isomorphic to the
direct product $(k^\times)^l \times \{\pm 1 \}$, where by $\{ \pm 1\}$
we mean the multiplicative group of square roots of unity. So the
character group $X(\TO)$ is isomorphic to $\Z^l \times (\Z/2\Z)$, and
thus we will identify $X(\TO)$ with $\Z^l \times (\Z/2\Z)$.

There is a group homomorphism $\pi \colon X(T) \to X(\TO)$ given by
restriction: $\pi(\lambda) = \lambda_{\mid \TO}$ for $\lambda \in
X(T)$. One easily checks that in this case, given a character $\xi
\in X(\TO)$, one can extend it to a character $\lambda \in X(T)$ such
that $\lambda_{\mid \TO} = \xi$. It follows that the map $\pi$ is
surjective.

In terms of the identifications $X(T) = \Z^n$ and $X(\TO) = \Z^l
\times (\Z/2\Z)$, the map $\pi$ is given by the rule
\[
  (\lambda_1, \dots, \lambda_n) \mapsto (\lambda_1 - \lambda_{1'},
\dots, \lambda_{l}-\lambda_{l'}, \overline{\lambda}_{l+1})
\]
where $\overline{m}$ denotes the image of an integer $m$ under the
natural quotient map $\Z \to \Z/2\Z$.  

We next consider how to characterize the image of the set
$\Lambda(n,r)$ under the map $\pi$. This case is a bit different from
the even orthogonal case, because of the presence of the $\Z/2\Z$ term
in the image of $\pi$. Note, however, that for any $\lambda \in
\Lambda(n,r)$, the last component $\overline{\lambda}_{l+1}$ of
$\pi(\lambda)$ is uniquely determined by the preceding entries in
$\pi(\lambda)$, as follows.

\begin{lem*}
  When $n = 2l+1$, suppose that $\lambda \in \Lambda(n,r)$ and put
  $t:= (\lambda_1-\lambda_{1'}) + \cdots +
  (\lambda_{l}-\lambda_{l'})$. Then $r-t \equiv \lambda_{l+1}
  \pmod{2}$.
\end{lem*}

\begin{proof}
Since $\lambda_1 + \cdots + \lambda_n = r$, it follows by a simple
calculation that $r-t = 2(\lambda_{1'} + \cdots + \lambda_{l'}) +
\lambda_{l+1}$, and the result follows.
\end{proof}

Thanks to the lemma, we may as well pay attention only to the first
$l$ parts of the image of some $\lambda \in \Lambda(n,r)$ under $\pi$.
Let us write $\Lamdot(l,r)$ for the set of all $(\xi_1, \dots, \xi_l)$
such that $(\xi_1, \dots, \xi_l, \varepsilon) \in \pi(\Lambda(n,r))$.
Then we have a bijection $\pi(\Lambda(n,r)) \to \Lamdot(l,r)$, given
by $(\xi_1, \dots, \xi_l, \varepsilon) \mapsto (\xi_1, \dots, \xi_l)$.
The inverse map is given by $(\xi_1, \dots, \xi_l) \mapsto (\xi_1,
\dots, \xi_l, \varepsilon)$, where $\varepsilon$ is the mod $2$
residue of $r - \xi_1 - \cdots - \xi_l$. This leads to the following
characterization of $\Lamdot(l,r)$.

\begin{prop}
  When $n = 2l+1$, the image of the set $\Lambda(n,r)$ under the map
  $\pi$ may be identified with the set $\Lamdot(l,r)$ consisting of
  all $\xi = (\xi_1, \dots, \xi_l) \in \Z^l$ such that $|\xi_1| +
  \cdots + |\xi_l| = r-s$, where $0 \le s \le r$.
\end{prop}

\begin{proof}
If $(\xi_1, \dots, \xi_l, \ep) = \pi(\lambda)$ for $\lambda \in
\Lambda(n,r)$ then $|\xi_1|+\cdots +|\xi_l|$ satisfies the condition
\[
  |\xi_1|+\cdots +|\xi_l| = \epsilon_1(\lambda_1-\lambda_{1'}) +
  \cdots + \epsilon_l(\lambda_l-\lambda_{l'})
\]
where for each $i=1, \dots, l$ the sign $\epsilon_i$ is defined to be
$1$ if $\lambda_i \ge \lambda_{i'}$ and $-1$ otherwise. This is just a
signed sum of the parts of $\lambda$, excluding the $(l+1)$st part
$\lambda_{l+1}$. Thus $|\xi_1|+\cdots +|\xi_l| = r-s$ where $0 \le s
\le r$. This proves the necessity of the condition for membership in
$\Lamddot(l,r)$.

It remains to prove the sufficiency of the condition. Given $\xi \in
\Z^l$ satisfying the condition $|\xi_1| + \cdots + |\xi_l|=r-s$, where
$0 \le s \le r$, we define a corresponding $\lambda \in
\Lambda(n,r-s)$ as follows: put $\lambda_i = \xi_i$ if $\xi_i > 0$,
put $\lambda_{i'} = -\xi_i$ if $\xi_i < 0$, put $\lambda_{l+1}=s$, and
put all the other entries of $\lambda = (\lambda_1, \dots, \lambda_n)$
to zero.  Then it easily checked that $\pi(\lambda)$ identifies with
$\xi$ under the correspondence $(\xi_1,\dots,\xi_l,\ep) \to
(\xi_1,\dots,\xi_l)$.
\end{proof}

For each $\xi \in \Lamdot(l,r)$, the fiber $\pi^{-1}(\xi)$ may be
computed as follows. Let $|\xi_1| + \cdots + |\xi_l|=r-s$, where $0
\le s \le r$, and let $\lambda$ be defined in terms of $\xi$ as in the
second paragraph of the proof of the proposition. For each integer $t$
such that $0 \le 2t \le s$, let $\mu$ be the same as $\lambda$ except
that $\lambda_{l+1}=s$ is replaced by $s-2t$. Then for any $\nu \in
\Lambda(l,t)$ we get a member 
\[
  (\mu_1+\nu_1, \dots, \mu_l+\nu_l, s-2t, \mu_{l'}+\nu_l, \dots,
\mu_{1'}+\nu_1)
\]
of the fiber $\pi^{-1}(\xi)$. Thus, the fiber in this case has
cardinality given by the sum $\sum_{0 \le 2t \le s} |\Lambda(l,t)|$.
By grouping terms in the direct sum decomposition
\eqref{eq:wtspaceLambda} according to the fibers we obtain
\begin{equation}\label{eq:wtspace2ell+1} \textstyle
   (k^n)^{\otimes r} = V^{\otimes r} = \bigoplus_{\xi \in
    \Lamdot(l,r)} \big(\bigoplus_{\lambda \in \pi^{-1}(\xi)}
  M^\lambda \big) = \bigoplus_{\xi \in \Lamdot(l,r)} N^{\xi}
\end{equation}
where we define $N^{\xi}$ for any $\xi \in \Lamdot(l,r)$ by
$N^{\xi}:= \bigoplus_{\lambda \in \pi^{-1}(\xi)} M^{\lambda}$.

The $N^{\xi}$ are just the weight spaces under the action of the
abelian group $\TO$, so \eqref{eq:wtspace2ell+1} gives the weight
space decomposition of tensor space as a $\TO$-module.

Since the actions of $\O_n(k)$ and $\B_r(n)$ commute, it is clear
that each weight space $N^{\xi}$ for $\xi\in \Lamdot(l,r)$ is
a right $\B_r(n)$-module. Hence \eqref{eq:wtspace2ell+1} is a
decomposition of tensor space $(k^n)^{\otimes r}$ as a
$\B_r(n)$-module, and we have achieved our goal in the case $n =
2l+1$.

The Weyl group $\dot{W}$ of $\O_n(k)$ in the case $n=2l+1$ is the
same as in the case $n=2l$; it is isomorphic to the semidirect
product $\{\pm 1\}^l \rtimes \Sym_l$, the group of signed
permutations on $l$ letters. We can realize $\dot{W}$ as a subgroup
of $\O_n(k)$ in this case as well, by taking the intersection of $W$
with $\O_n(k)$.  A given $w \in W$ lies within this intersection if and
only if the condition $(e_{w^{-1}(i)} , e_{w^{-1}(j)}) = (e_i, e_j)$
holds for all $i,j$. Thus, $\dot{W}$ is the set of $w \in W$ such that
\begin{equation}\label{eq:W-criterion-2}
  \delta_{w^{-1}(i), w^{-1}(j)'} = \delta_{i,j'} \qquad \text{for all }
  i,j = 1, \dots, n.
\end{equation}
Thus, for $w \in W$ to belong to $\dot{W}$, it is necessary that
$w^{-1}(l+1) = l+1$, or, equivalently, $w(l+1)=l+1$.  Then it is easy
to check by direct calculation that for any given $\sigma \in
\Sym_l$, if we define a corresponding $w_\sigma\in W$ such that
\[
  w_\sigma(i) = 
  \begin{cases}
  \sigma(i) & \text{if } 1 \le i \le l\\
  \sigma(i') & \text{if } l+1 \le i \le 2l  
  \end{cases}
\]
then $\sigma$ satisfies the condition
\eqref{eq:W-criterion-2}. Furthermore, the transposition $\tau_i$ that
interchanges $i$ with $i'$ also satisfies \eqref{eq:W-criterion-2}, and
thus $\dot{W}$ may be identified with the subgroup of $W$ generated by
the $w_\sigma$ ($\sigma \in \Sym_l$) and the $\tau_i$ ($i = 1, \dots,
l$). 

This subgroup acts on $\Lambda(n,r)$ by restriction of the action of
$W$. This induces a corresponding action of $\dot{W}$ on the set
$\Lamdot(l,r)$, such that $\dot{w}(\xi) = \pi( w(\lambda) )$ if $w
\in W$ corresponds to $\dot{w} \in \dot{W}$ and $\xi =
\pi(\lambda)$.  Since $\tau_i$ sends $\xi = (\xi_1, \dots,
\xi_l)$ to $(\xi_1, \dots, \xi_{i-1}, -\xi_i,
\xi_{i+1}, \dots, \xi_l)$, and $w_\sigma$ sends $\xi$ to
$\sigma(\xi) = (\xi_{\sigma^{-1}(1)}, \dots,
\xi_{\sigma^{-1}(l)})$, it follows that $\dot{W}$ acts on the set
$\Lamdot(l,r)$ by signed permutations.

Thus, a fundamental domain for this action is the set
$\Lamdot^+(l,r)$ consisting of all $\xi \in \Lamdot(l,r)$ such
that $\xi_1 \ge \xi_2 \ge \cdots \ge \xi_l \ge 0$. We call
elements of this set dominant orthogonal weights. So, in other words,
each orbit of $\Lamdot(l,r)$ contains a unique dominant orthogonal
weight. Notice that a dominant orthogonal weight is the same as a
partition of not more than $l$ parts.

\begin{prop}
  For any $\dot{w} \in \dot{W}$, $\xi \in \Lamdot(l,r)$, the right
  $\B_r(n)$-modules $N^{\xi}$ and $N^{\dot{w}(\xi)}$ are isomorphic.
\end{prop}

The proof is similar to the proof of Proposition \ref{prop:3}.

Hence, when studying properties of the modules $N^{\xi}$, we may as
well confine our attention to the ones indexed by dominant orthogonal
weights; i.e., partitions. In the decomposition
\eqref{eq:wtspace2ell+1} each summand is isomorphic to some $N^{\xi}$
for some $\xi$ such that $\xi$ is a partition of $r - s$ into not
more than $l$ parts, for some non-negative integer $s \le r$. It
is easy to see that all such possibilities actually occur as direct
summands in the decomposition \eqref{eq:wtspace2ell+1}.

\section{The $\B_r(-n)$ decomposition of $(k^n)^{\otimes r}$ 
in the skew-symmetric case, where $n=2l$}\label{sec:skew}\noindent
In this case we assume throughout that $n=2l$, and let the field $k$
be arbitrary.  We define $\Sp_n(k)$ to be the group of isometries of
$V=k^n$ with respect to the skew-symmetric form $(\ ,\ )$ defined in
Section \ref{sec:main}.  Then the action of $\B_r(-n)$ on tensor space
$V^{\otimes r} = (k^n)^{\otimes r}$, defined in Section
\ref{sec:Brauer}, commutes with the natural action of $\\Sp_n(k)$
(given by restricting the action of $\GL_n(k)$).

Let $\TO$ be the abelian subgroup of $\Sp_n(k)$ consisting of the
diagonal matrices in $\O_n(k)$. Thus, a diagonal matrix
$\text{diag}(t_1, \dots, t_n) \in \GL_n(k)$ belongs to $\TO$ if and
only if 
\begin{equation}\label{eq:TO-Sp}
t_i t_{i'} = 1 \quad\text{ for all $i = 1, \dots, n$.} 
\end{equation}
As before, the description of $\TO$ in \eqref{eq:TO-Sp} shows in this
case that $\TO$ is isomorphic to $(k^\times)^l$, and $X(\TO)$ is
isomorphic to $\Z^l$, so we identify $X(\TO)$ with $\Z^l$.

The group homomorphism $\pi \colon X(T) \to X(\TO)$ given by
restriction is surjective, for the same reason as before.  In terms of
the identifications $X(T) = \Z^n$ and $X(\TO) = \Z^l$, the map $\pi$
is given by the rule
\[
  (\lambda_1, \dots, \lambda_n) \mapsto (\lambda_1 -
   \lambda_{1'}, \dots, \lambda_{l}-\lambda_{l'}). 
\]
The image$\Lamddot(l,r)$ of the set $\Lambda(n,r)$ under the map $\pi$
has the same characterization as in the even symmetric case.

\begin{prop}
  When $n = 2l$, the image of the set $\Lambda(n,r)$ under the map
  $\pi$ is the set $\Lamddot(l,r)$ of all $\xi = (\xi_1, \dots, \xi_l)
  \in \Z^l$ such that $|\xi_1| + \cdots + |\xi_l| =r-2s$, where $0 \le
  2s \le r$.
\end{prop}

The proof is the same as in the even symmetric case; see the proof of
Proposition \ref{prop:2}.

The fiber $\pi^{-1}(\xi)$ for $\xi \in \Lamddot(l,r)$ has in this case
the same description as in the even symmetric case; see the remarks
following the proof of Proposition \ref{prop:2}.  By grouping terms in
the direct sum decomposition \eqref{eq:wtspaceLambda} according to the
fibers we obtain
\begin{equation}\label{eq:wtspace-Sp} \textstyle
   (k^n)^{\otimes r} = V^{\otimes r} = \bigoplus_{\xi \in
    \Lamddot(l,r)} \big(\bigoplus_{\lambda \in \pi^{-1}(\xi)}
  M^\lambda \big) = \bigoplus_{\xi \in \Lamddot(l,r)} N^{\xi}
\end{equation}
where we define $N^{\xi}$ for any $\xi \in \Lamddot(l,r)$ by
$N^{\xi}:= \bigoplus_{\lambda \in \pi^{-1}(\xi)} M^{\lambda}$.

The $N^{\xi}$ are just the weight spaces under the action of the
abelian group $\TO$, so \eqref{eq:wtspace-Sp} gives the weight space
decomposition of tensor space as a $\TO$-module.

Since the actions of $\Sp_n(k)$ and $\B_r(-n)$ commute, it is clear
that each weight space $N^{\xi}$ for $\xi\in \Lamddot(l,r)$ is a right
$\B_r(-n)$-module. Hence \eqref{eq:wtspace-Sp} is a decomposition of
tensor space $(k^n)^{\otimes r}$ as a $\B_r(n)$-module.

The Weyl group $\dot{W}$ of $\Sp_n(k)$ is again isomorphic to the
semidirect product $\{\pm 1\}^l \rtimes \Sym_l$, the group of signed
permutations on $l$ letters. Again, the group $\dot{W}$ acts on the
set $\Lamddot(l,r)$ by signed permutations.  Thus, a fundamental
domain for this action is the set $\Lamddot^+(l,r)$.

\begin{prop}\label{prop:7}
  For any $\dot{w} \in \dot{W}$, $\xi \in \Lamddot(l,r)$, the
  right $\B_r(n)$-modules $N^{\xi}$ and $N^{\dot{w}(\xi)}$ are
  isomorphic.
\end{prop}

The proof is similar to the proof of Proposition \ref{prop:3}. 

Hence, when studying properties of the modules $N^{\xi}$, we may as
well confine our attention to the ones indexed by dominant orthogonal
weights; i.e., partitions. In the decomposition \eqref{eq:wtspace-Sp}
each summand is isomorphic to some $N^{\xi}$ for some $\xi$ such that
$\xi$ is a partition of $r - 2s$ into not more than $l$ parts, for
some non-negative integer $s \le r/2$. As before, it is easy to see
that all such possibilities actually occur as direct summands in
\eqref{eq:wtspace-Sp}.


\begin{thebibliography}{99}\frenchspacing\raggedright\small
\bibitem{BD} D. Benson and S. Doty, Schur--Weyl duality over finite
  fields, \emph{Arch. Math.} (Basel) 93 (2009), 425--435.
\bibitem{Brauer} R. Brauer, On algebras which are connected with the
  semisimple continuous groups, \emph{Ann. of Math.} (2) 38 (1937),
  857--872.
\bibitem{Brown} W. Brown, An algebra related to the orthogonal
  group, \emph{Michigan Math. J.} 3 (1955), 1--22.
\bibitem{DDH} R. Dipper, S. Doty, and J. Hu, Brauer algebras,
  symplectic Schur algebras and Schur-Weyl duality,
  \emph{Trans. Amer. Math. Soc.} 360 (2008), 189--213
  (electronic).
\bibitem{Doty:SWD} S. Doty, Schur-Weyl duality in positive
  characteristic, \emph{Representation theory}, 15--28,
  Contemp. Math., 478, Amer. Math. Soc., Providence, RI, 2009.
\bibitem{DH} S.Doty and J. Hu, Schur-Weyl duality for orthogonal
  groups, \emph{Proc. Lond. Math. Soc.} (3) 98 (2009), 679--713.
\bibitem{James} G. James, \emph{The representation theory of the
  symmetric groups}, Lecture Notes in Mathematics, 682, Springer,
  Berlin, 1978.
\bibitem{Fulton} W. Fulton, \emph{Young Tableaux, With applications
  to representation theory and geometry}, London Mathematical Society
  Student Texts, 35, Cambridge University Press, Cambridge, 1997.
\bibitem{Green} J.A. Green, \emph{Polynomial Representations of
  $\GL_n$}, Lecture Notes in Mathematics, 830. Springer-Verlag,
  Berlin-New York, 1980.
\bibitem{HP} A.E. Henke and R. Paget, Brauer algebras with parameter
  $n=2$ acting on tensor space, \emph{Algebr. Represent. Theory} 11
  (2008), 545--575.
\bibitem{Schur} I. Schur, \"{U}ber die rationalen Darstellungen der
  allgemeinen linearen Gruppe (1927), Gesammelte Abhandlungen, Band III
  (German), 68--85, herausgegeben von Alfred Brauer und Hans Rohrbach,
  Springer-Verlag, Berlin/New York, 1973.
\bibitem{Weyl} H. Weyl, \emph{The Classical Groups; Their Invariants
  and Representations}, Princeton University Press, Princeton, N.J.,
  1939.
\end{thebibliography}
\end{document}